\documentclass[12pt,reqno]{amsart}
\usepackage{amsmath,bm}
\usepackage{amsfonts}
\usepackage{amssymb}
\usepackage{amsthm}
\usepackage{amscd}

\usepackage{amsmath}
\usepackage{amssymb}
\usepackage{amsthm}
\usepackage{graphics}
\usepackage{eucal}
\usepackage{mathrsfs}

\usepackage{color}
\theoremstyle{plain}
\newtheorem{theorem}{Theorem}[section]
\newtheorem{proposition}[theorem]{Proposition}
\newtheorem{lemma}[theorem]{Lemma}
\newtheorem{corollary}[theorem]{Corollary}

\theoremstyle{definition}

\newcommand{\R}{{\mathbb{R}}}
\newcommand{\Z}{{\mathbb{Z}}}
\newcommand{\N}{{\mathbb{N}}}

\newcommand{\T}{{\mathbb{T}}}

\newcommand{\GL}{\operatorname{GL}}

\newcommand{\Mat}{\operatorname{Mat}}

\newcommand{\Id}{\operatorname{Id}}

\newcommand{\dist}{\operatorname{dist}}

\begin{document}

\title[Simultaneous dense and nondense orbits]{Simultaneous dense and nondense orbits for noncommuting toral endomorphisms}

\begin{abstract}
 Let $S$ and $T$ be hyperbolic endomorphisms of $\T^d$ with the property that the span of the subspace contracted by $S$ along with the subspace contracted by $T$ is $\R^d$.  We show that the Hausdorff dimension of the intersection of the set of points with equidistributing orbits under $S$ with the set of points with nondense orbit under $T$ is full.  In the case that $S$ and $T$ are quasihyperbolic automorphisms, we prove that the Hausdorff dimension of the intersection is again full when we assume that $\R^d$ is spanned by the subspaces contracted by $S$ and $T$ along with the central eigenspaces of $S$ and $T$. 
 \end{abstract}


\author{Beverly Lytle}
\author{Alex Maier}

\maketitle
\tableofcontents

\section{Introduction}

Questions about the size of the sets with points satisfying certain conditions on their orbits has received much attention in recent years.  Most notably, there are many results in the area of diophantine approximation (for example, \cite{BET, BFK, Dani1, D2, EKL, KM} and further \cite{F, K2, KW1, KW2}).  Many of diophantine properties can be expressed in terms of the behavior of an orbit of a point in a homogeneous space.  For example, the notion of bad approximability can be interpreted in terms of bounded orbits, while very well approximable objects correspond to divergent orbits.  However, these types of results generally refer to the behavior of a point under one transformation or flow.  In this work we attempt to address the simultaneous behavior of a point under two different transformations.  

Let us begin with introducing our setting.  Let $\T^d = \R^d / \Z^d$ be the $d$ dimensional torus, and let $S$ be an ergodic endomorphism of $\T^d$.  Then the transformation $S$ can be realized as an integer matrix with no eigenvalue a root of unity.  In this case $S$ is called quasihyperbolic, and in the case that $S$ has no eigenvalue of modulus 1, $S$ is called hyperbolic.  It is well known that, since $S$ is ergodic, the set of points, denoted $Eq(S)$, whose orbits equidistribute under $S$ is a full measure set of $\T^d$.  Consequently, the set of points, denoted ND(S), which have nondense orbit under $S$ has measure zero as it is contained in the complement, $NEq(S)$,  of $Eq(S)$.  While the set $ND(S)$ is small in a measure theoretic sense, it does have full Hausdorff dimension \cite{BFK}.  Indeed, it satisfies a stronger property known as winning.  (Definitions of these terms are given in Section 2.)   All this is to say that in the compact case, much is known about the behavior of points under a single transformation.

Let us then introduce $T$, a second ergodic endomorphism of $\T^d$.  Clearly, as the intersection of full measure sets, $Eq(S)\cap Eq(T)$ is a full measure set, hence of full Hausdorff dimension.  Further, it is a property of winning sets, that $ND(S)\cap ND(T)$ is again winning, and therefore of full Hausdorff dimension.  Now what can be said of $Eq(S) \cap ND(T)$?  In \cite{BET}, Bergelson, Einsiedler and Tseng showed that for two commuting hyperbolic toral endomorphisms $S$ and $T$ which generate an algebraic $\Z^2$ action without rank 1 factors, $\dim(ND(T) \cap D(S)) $ is greater than the dimension of the unstable manifold determined by $T$.  (Here $D(S)$ is the set of points of $\T^d$ which have dense orbit, a slightly larger set than $Eq(S)$, and $\dim$ refers to Hausdorff dimension.)  Their work relies on the measure rigidity theorems of Einsiedler and Lindenstrauss \cite{EL} and that the generalized eigenspaces of $S$ and $T$ are aligned (since the maps commute).  In this article, we wish to determine the size of the set $Eq(S) \cap ND(T)$ in the situation that the generalized eigenspaces of $S$ and $T$ are not aligned in the following sense.


\begin{theorem}\label{Thm3}
Let $\T^d =\R^d / \Z^d$ be the $d$ dimensional torus.  Let $S$ and $T$ be quasihyperbolic automorphisms of $\T^d$. Denote by $\mathfrak{s}_-$ (resp. $\mathfrak{t}_-$) the subspace contracted by $S$ (resp. by $T$) and denote by $\mathfrak{s}_0'$ (resp. $\mathfrak{t}_0'$) the sum of the eigenspaces of $S$ (resp. of $T$) of eigenvalue of modulus 1.  Assume that $\R^d$ is spanned by $\mathfrak{s}_-\oplus \mathfrak{s}_0'$ and $\mathfrak{t}_-\oplus \mathfrak{t}_0'$.  Then 
\[
\dim( Eq(S)  \cap ND(T)) = d.
\]
\end{theorem}

We have two versions of the proofs of Theorem~\ref{Thm3} and two corollaries each following from one version of the proof and the properties of winning sets.  We explicate the corollaries following from the proof of Theorem~\ref{Thm3}.

\begin{corollary}\label{Cor3}
Let $S$ be a quasihyperbolic endomorphism of $\T^d$.  Let $\{T_k\}_{k}$ be a countable collection of  quasihyperbolic automorphisms such that each, when paired with $S$, satisfy the conditions of  Theorem~\ref{Thm3} on the generalized eigenspaces of $S$ and $T$.  Then 
\[
\dim( Eq(S)  \cap  \bigcap_{k}ND(T_k)) = d.
\] 
\end{corollary}

\begin{corollary}\label{Cor4}
Let $T$ be a quasihyperbolic automorphism of $\T^d$.  Let $\{S_k\}_{k}$ be a countable collection of  quasihyperbolic endomorphisms such that each, when paired with $T$, satisfy the conditions of  Theorem~\ref{Thm3}  on the generalized eigenspaces of $T$ and $S$.  Then 
\[
\dim(\bigcap_k Eq(S_k)\cap ND(T)) = d.
\]  
\end{corollary}

We remark that with the methods we use it is not possible to say anything about sets of the form $ Eq(S_1) \cap Eq(S_2) \cap ND(T_1) \cap ND(T_2) $.

We remark further that if $S$ and $T$ are non-commuting ergodic automorphisms of the 2-torus, then the hypotheses of the theorem are automatically satisfied.  Indeed, if $S$ and $T$ share one eigenvector, then, by considering the underlying splitting field for the characteristic polynomial of $T$ (and also of $S$), one sees that the other eigenspace of $T$ and hence also of $S$ must be the conjugate.  Thus, if $S$ and $T$ share one eigenspace, then they share both, whence they commute.  (See~\cite{Tseng}.)





The outline of the paper is as follows.  In Section 2, we give some preliminaries on the basics of winning sets and ergodic toral endomorphisms.  The third section is devoted to the proof of the theorem in the case that $S$ and $T$ are hyperbolic, exposing some structure of the sets $ND(T)$ and $Eq(S)$, taking advantage of that structure to prove a lemma about winning sets, and a recollection and application of Marstrand's slicing theorem.  Section 4 contains the proofs of Theorem~\ref{Thm3} which in structure is the same as the proof contained in Section 3 but  to complete the details, requires some spectral theory, the notion of joinings of dynamical systems and topological entropy.

\section{Preliminaries}
Let $\T^d= \R^d \backslash \Z^d$ be the $d$-dimensional torus equipped with a Haar measure $m$ normalized to be a probability measure.  We also choose $\Vert \cdot \Vert$ to be the maximum norm on $\R^d$.  Let $S$ be an ergodic toral endomorphism, that is, $S$ can be given as a nonsingular $d\times d$ matrix with integer entries and each eigenvalue of this matrix is not a root of unity.  This property is often named quasihyperbolicity.  It will be convenient to think of $S$ as both a transformation on $\R^d$ and on $\T^d$. Hopefully, on which space $S$ is acting will be clear from context.  

We define the following sets related to the $\Z$-action of $S$ on $\T^d$:
\begin{align*}
D(S) &= \{ x\in \T^d ~:~ \overline{ \{S^nx\}_{n \in \N_0} } = \T^d \}, \text{ and}\\ 
Eq(S) &= \{ x \in \T^d ~:~ \frac{1}{N}\sum_{i=0}^{N-1} f ( S^n x ) \underset{N\to \infty}{\longrightarrow} \int f\, dm ~   \forall f \in C(\T^d) \}, 
\end{align*}
and their complements $ND(S) = \T^d \setminus D(S)$ and $NEq(S) = \T^d \setminus Eq(S)$.
Certainly, we have $Eq(S) \subset D(S)$.  Applying the Pointwise Ergodic Theorem of Birkhoff, we have the following fact:
\begin{proposition}
With respect to $m$, almost every point $x\in\T^d$ equidistributes under $S$.  Thus, $1 = m(Eq(S)) = m(D(S))$.
\end{proposition}
It follows then that $NEq(S)$ and $ND(S)\subset NEq(S) $ are sets of measure zero.  However, $ND(S)$ is ``large" in another sense.  To understand this notion of ``large", we recall the definition and some properties of winning sets in the sense of Schmidt.  

In \cite{S1}, W. Schmidt introduced a game and the definition of winning along with a few properties of winning sets.  The game is played on $(X,\dist)$, a complete metric space.  Denote by $B(x,r)$ the closed metric ball around a point $x$ of radius $r$.  The setup of the two player game is given by two parameters $0<\alpha,\beta<1$, a set $S\subset X$, and the choice of one of the players, let's call him Bob, of a ball $B_0 = B(x_0, \rho)$.  The first round begins with the other player,  called Alice, choosing a center point of a ball $y_1$ such that $A_1 = B(y_1,\rho\alpha)\subset B_0$.  Bob chooses the next center point of a ball $x_1$ such that $B_1= B(x_1, \rho\alpha\beta) \subset A_1$.  This procedure is iterated with the $n$th round of the game beginning with Alice choosing a point $y_n$ with $A_n=B(y_n,\rho\alpha(\alpha\beta)^{n-1})\subset B_{n-1}$, and continuing with Bob choosing a point $x_n$ satisfying $B_n = B(x_n,\rho(\alpha\beta)^n) \subset A_n$.  At the end of the game, there remains one point $x_\infty \in \bigcap B_n$.  If $x_\infty \in S$, then Alice wins.  If Alice can always find a winning strategy independent of the moves of Bob, the set $S$ is $(\alpha,\beta)$-winning.  If there exists $\alpha$ such that $S$ is $(\alpha, \beta)$-winning for all $\beta>0$, then $S$ is an $\alpha$-winning set, which may be shortened to $S$ is a winning set.  

Winning sets have a number of useful properties for computing Hausdorff dimension.  Schmidt showed in \cite{S1} that winning sets within $X= \R^d$ have Hausdorff dimension $d$ (although more general statements exist~\cite{F1,F2}).  Moreover he showed in \cite{S1} that for a countable collection $\{S_i\}$ of $\alpha_i$-winning sets with $\inf \alpha_i = \alpha_0 >0$, the intersection $\bigcap_i S_i$ is $\alpha_0$-winning.

The following theorem of Broderick, Fishman and Kleinbock \cite{BFK} will be useful. 
\begin{theorem}\label{thm:BFK}
For any $M \in \GL_d(\R) \cap \Mat_{d,d}(\Z)$ and for any $y \in \T^d$, the set $E(M,y)$ is $1/2$-winning in $\T^d$.
\end{theorem}
Here $E(M,y) = \{ x \in \T^d ~:~ y \notin \overline{ \{M^n x\} } \}$.  Clearly for any $y\in T^d$, $E(S,y)$ is a subset of $ND(S)$.  Thus we have that:
\begin{corollary}
For any ergodic toral endomorphism $S$, the set $ND(S)$ is $1/2$-winning, and hence $\dim(ND(S))=d$.  
\end{corollary}

To finish we establish some notation regarding the eigenspaces of $S$.  Viewing $S$ as a transformation of $\R^d$ (which maps $\Z^d$ into itself), we have the following decomposition
\[
\R^d = \mathfrak{s}_- \oplus \mathfrak{s}_0 \oplus \mathfrak{s}_+,
\]
where $\mathfrak{s}_-$ is the subspace corresponding to the eigenvalues of $S$ of modulus less than 1, $\mathfrak{s}_+$ is the subspace corresponding to the eigenvalues of $S$ of modulus greater than 1, and $\mathfrak{s}_0$ is the subspace corresponding to all other eigenvalues of $S$.  For a second ergodic toral endomorphism  $T$, we will use similar notation for the decomposition
\[
\R^d = \mathfrak{t}_- \oplus \mathfrak{t}_0 \oplus \mathfrak{t}_+.
\]

When speaking about a measure we always mean the Lebesgue measure. Of which dimension should be clear from context, since we are often talking about subspaces.

\section{Hyperbolic toral endomorphisms}
In this section, we specialize to the case of hyperbolic (and of course ergodic) toral endomorphisms, and give a proof of the main theorem in this case.   Specifically, we prove:

\begin{theorem}\label{Thm1}
Let $\T^d =\R^d / \Z^d$ be the $d$-dimensional torus.  Let $S$ and $T$ be hyperbolic endomorphisms of $\T^d$ and let $\mathfrak{s}_-$ be the subspace contracted by $S$ and $\mathfrak{t}_-$ the subspace contracted by $T$.  Assume that both $\mathfrak{s}_-$ and $\mathfrak{t}_-$ are nontrivial and that they span $\R^d$. Then 
\[ \dim( Eq(S)  \cap ND(T) )= d.\]
\end{theorem}

We give two proofs of this theorem.  By properties of winning sets, we will have two immediately corollaries, one from each version of the proof of the theorem.

\begin{corollary}\label{Cor1}
Let $S$ be a hyperbolic endomorphism of $\T^d$.  Let $\{T_k\}_{k }$ be a countable collection of hyperbolic endomorphisms such that each, when paired with $S$, satisfy the conditions on the generalized eigenspaces of  Theorem~\ref{Thm1} on $S$ and $T$.  Then 
\[
 \dim( Eq(S)  \cap  \bigcap_{k}ND(T_k)) = d. 
 \]
\end{corollary}

\begin{corollary}\label{Cor2}
Let $T$ be a hyperbolic endomorphism of $\T^d$.  Let $\{S_k\}_{k }$ be a countable collection of hyperbolic endomorphisms such that each, when paired with $T$, satisfy the conditions on the generalized eigenspaces of  Theorem~\ref{Thm1} on $T$ and $S$.  Then 
\[
\dim(\bigcap_k Eq(S_k)\cap ND(T) )= d.
\]  
\end{corollary}

Let $S$ and $T$ be endomorphisms of $\T^d$ such that no eigenvalue of $S$ or of $T$ is of modulus 1 or 0.  This means that $\mathfrak{s}_0 = \mathfrak{t}_0 = 0$ and  we have
\[
\R^d = \mathfrak{s}_- \oplus \mathfrak{s}_+ = \mathfrak{t}_- \oplus \mathfrak{t}_+.
\]
We also assume that each of the subspaces $\mathfrak{s}_-$ and $\mathfrak{t}_-$ are nontrivial, and that $\R^d$ is spanned by these subspaces.  We choose nontrivial subspaces $\mathfrak{s} \subset \mathfrak{s}_-$ and $\mathfrak{t} \subset \mathfrak{t}_-$ such that $\R^d = \mathfrak{s} \oplus \mathfrak{t}$.  Define $\pi_{\mathfrak{s}}: \R^d \to \mathfrak{s}$ to be the projection parallel to $\mathfrak{t}$ and similarly define $\pi_{\mathfrak{t}}: \R^d \to \mathfrak{t}$ to be the projection parallel to $\mathfrak{s}$.  In this section, there is no harm in identifying $\mathfrak{s}$ and $\mathfrak{t}$  and their elements with their images in $\T^d$.

The following lemma is a key observation for the proof of this special case.
\begin{lemma}
Let $x\in\T^d$.  Then for any $y\in \mathfrak{s}$, the orbits of the points $x$ and $x+y$ under $S$ are both dense in $\T^d$ or both not dense in $\T^d$.  Furthermore, if one of the orbits equidistributes, then both do.  An analogous statement holds for $y\in\mathfrak{t}$ and $T$.
\end{lemma}

\begin{proof}
Since $y$ is contracted by $S$ we have that
\[
\Vert S^n (x) - S^n (x+y) \Vert = \Vert S^n(y) \Vert \to 0 \text{ as } n\to \infty.
\]
Thus the tails of the orbits of $x$ and of $x+y$ are nearing identical, and so the lemma quickly follows. 
\end{proof}

This means that for any $x\in \T^d$, every point of the leaf $x + \mathfrak{s}$ has the same behavior asymptotically, either nondense or dense (and perhaps equidistributing).  Therefore $Eq(S)$ can be written as the disjoint union of sets of the form $x+\mathfrak{s}$ where $x$ ranges over $\mathfrak{t} \cap Eq(s) = \pi_{\mathfrak{t}}(Eq(S))$.  Similarly, we have a disintegration of $D(S)$ into the sets $x + \mathfrak{s}$ for $x\in \mathfrak{t} \cap D(S) = \pi_{\mathfrak{t}}(D(S))$, and a disintegration of $ND(T)$ into the sets $x+\mathfrak{t}$ for $x\in\mathfrak{s} \cap ND(T) = \pi_{\mathfrak{s}}(ND(T))$.  This type of foliated structure leads to the next lemma.

\begin{lemma}\label{lem:FullMeas}
For almost every $x\in \mathfrak{t}$, we have that $x$ is in $Eq(S)$.  In other words, $Eq(S)\cap\mathfrak{t}$ is a full measure set within $\mathfrak{t}$ (with respect to Lebesgue measure on $\mathfrak{t}$).  An analogous statement holds for $D(S)$.
\end{lemma}

\begin{proof}
As shown in the previous section, $Eq(S)$ is a full measure set, and each line parallel to $\mathfrak{s}$ is either completely contained in $Eq(S)$ or disjoint from it. Moreover every such line is transverse to $\mathfrak{t}$.  More formally, this is an application of the theorem of Fubini: Using the equalities
\[
1 = m(Eq(S)) = \int_{x\in\mathfrak{t}}\int_{x+\mathfrak{s}} \chi_{Eq(S)} = \int_{x\in\mathfrak{t} \cap Eq(s)} 1,
\]
we can conclude that $\mathfrak{t} \cap Eq(S)$ has full measure in $\mathfrak{t}$.

The same argument holds for $D(S)$.
\end{proof}

We would also like to say something about the size of the set $\mathfrak{s} \cap ND(T)$.  Since $ND(T)$ is a measure zero set, the same arguments do not apply.  We will instead have to appeal to the fact that it is a winning set and use the following lemma.

\begin{lemma}\label{lem:WinProj}
Let $(X,\dist)$ be a complete metric space with a subset $U\subset X$ admitting a product structure $U = V\times W$ (meaning $\dist$ restricted to $U$ is the maximum of the distance on $V$ and the distance on $W$).  Denote by $\pi_V$ the projection from $U$ to $V$.  Suppose that $A\subset V$ is such that $\pi_V^{-1}(A)$ is an $\alpha$-winning set in $U$.  Then $A$ is an $\alpha$-winning set in $V$.
\end{lemma}

\begin{proof}
Since $\pi_V^{-1}(A)$ is $(\alpha,\beta)$-winning for all $\beta>0$, Alice will simply use the strategy for winning the $(\alpha,\beta)$-game on $U$ to win the $(\alpha,\beta)$-game on $V$ by projecting her moves using $\pi_V$ to $V$.  Explicitly, suppose in round $n$ of an $(\alpha,\beta)$-game on $V$ Bob has chosen the ball $B(x_n,\rho_n)$.  Alice must choose a subball $A_{n+1} = B(y_{n+1}, \alpha\rho_n) \subset B(x_n, \rho_n)$.  She does this by considering $\pi^{-1}_V(B(x_n,\rho_n))$, which is a metric ball in $U$, employing the known winning strategy there to find $A_{n+1}' \subset \pi^{-1}_V(B(x_n,\rho_n))$, and setting $A_{n+1} = \pi_V(A'_{n+1})$.  Since the end point of the $(\alpha,\beta)$-game on $U$ is in $\pi_v^{-1} (A)$, the end point of the corresponding $(\alpha,\beta)$-game on $V$ is in $A$.
\end{proof}

This lemma leads to the following analogue of Lemma~\ref{lem:FullMeas} for $ND(T)$ and $\mathfrak{s}$.  

\begin{lemma}
The set $ND(T) \cap \mathfrak{s} = \pi_{\mathfrak{s}}(ND(T))$ is a 1/2-winning set in $\mathfrak{s}$.
\end{lemma}

\begin{proof}
This is an application of Lemma~\ref{lem:WinProj} with the result of Theorem~\ref{thm:BFK}.
\end{proof}

For the next step, we compute the Hausdorff dimension of $Eq(S) \cap ND(T)$.  For this we use Kleinbock and Margulis' version of the \linebreak Marstrand slicing theorem \cite{KM}:

\begin{theorem}\label{thm:MarSlice}
Let $M_1$ and $M_2$ be Riemannian manifolds, $A \subset M_1$, $B \subset M_1 \times M_2$. Denote by $B_a$ the intersection of $B$ with $\{a\} \times M_2$ and assume that $B_a$ is nonempty for all $a \in A$. Then
$$\dim (B) \geq \dim (A) + \inf_{a \in A} \dim (B_a).$$
\end{theorem}

Now we have all ingredients and are ready to prove Theorem \ref{Thm1}:

\begin{proof}[Proof of Theorem \ref{Thm1}]
For this theorem we give two proofs.  Both are needed to prove each of the corollaries.

Version 1:  Define $A = \mathfrak{t} \cap Eq(S)$.  This is a full measure set in $\mathfrak{t}$ (with respect to Lebesgue measure on $\mathfrak{t}$), and therefore has full Hausdorff dimension, that is, $\dim(A) = \dim(\mathfrak{t})$.  For every $a\in A$, define $B_a = (a + \mathfrak{s}) \cap ND(T)$.  Notice then that $B= \bigcup_{a\in A}B_a$ is equal to $Eq(s)\cap ND(T)$.  Due to the foliated structure of $ND(T)$, we have that $a+\pi_{\mathfrak{s}}(ND(T)) = B_a$.  Since the projection (and translate) of a 1/2-winning set is still 1/2-winning by Lemma~\ref{lem:WinProj}, $B_a$ is a winning subset of $a+\mathfrak{s}$ and so $\dim(B_a) = \dim(a+\mathfrak{s}) = \dim(\mathfrak{s})$.  Now we apply the above theorem to conclude
\[
\dim(Eq(S) \cap ND(T)) = \dim(B) \geq \dim(\mathfrak{t}) + \dim(\mathfrak{s}) = n.
\]
Since the left side of the inequality is bounded above by $n$, we have equality and our theorem is proven.

Version 2:  This time we reverse the roles of $\mathfrak{s}$ and $\mathfrak{t}$, and we define $A= \mathfrak{s} \cap ND(T)$.  This we have shown to be a 1/2-winning set in the subspace $\mathfrak{s}$ and therefore $\dim(A) = \dim(\mathfrak{s})$.  For every $a\in A$ define $B_a = (a+\mathfrak{t})\cap Eq(S)$.  Due to the foliated structure of $Eq(S)$, we have that $a+\pi_{\mathfrak{t}}(Eq(S)) = B_a$.  By Fubini's theorem, we know that $B_a$ has full measure within $a+\mathfrak{t}$, and so $\dim(B_a) = \dim(a+\mathfrak{t}) = \dim(\mathfrak{t})$.  Again we apply Theorem~\ref{thm:MarSlice} to conclude
\[
\dim(Eq(S) \cap ND(T)) = \dim(B) \geq \dim(\mathfrak{t}) + \dim(\mathfrak{s}) = n,
\]
where $B= \bigcup_{a\in A}B_a$ as before. Thus, the theorem is proven again.
\end{proof}

\begin{proof}[Proof of Corollary~\ref{Cor1}]
We extend the first version of the proof of Theorem~\ref{Thm1} to a proof of the corollary. Since we have more than two endomorphisms, we must be more flexible in the choice of the spaces $\mathfrak{s}$ and $\mathfrak{t}$. We define $\mathfrak{s} = \mathfrak{s}_-$.  Note that $\mathfrak{s} \neq \R^d$ because $S$ maps $\Z^d$ to itself and cannot contract in all directions simultaneously.  Now we choose an arbitrary subspace $\mathfrak{t}$ such that $\mathfrak{s} \oplus \mathfrak{t} = \R^d$. With this choice of $\mathfrak{t}$ we then take the same definition of the set $A$, that is $A= \mathfrak{t} \cap Eq(S)$.   For each $T_k$ we choose a subspace $\mathfrak{t}^k \subseteq \mathfrak{t}^k_-$ such that $\mathfrak{s} \oplus \mathfrak{t}^k = \R^d$, and for every $a\in A$, set $B_a^k = (a+\mathfrak{s}) \cap ND(T_k)$.  Using the subspaces $\mathfrak{t}_k$, we see as before that each the $B_a^k$ are $1/2$-winning.  Hence, $\bigcap_kB_a^k$ is $1/2$-winning as well.  We proceed in the proof by applying the Marstrand slicing theorem to the sets $A$ and $\bigcap_kB_a^k$.  
\end{proof}

\begin{proof}[Proof of Corollary~\ref{Cor2}]
 For the second corollary we extend the second version of the proof of Theorem~\ref{Thm2} in an analogous way:
We define $\mathfrak{t} = \mathfrak{t}_-$. Then we choose an arbitrary subspace $\mathfrak{s}$ such that we have $\mathfrak{s} \oplus \mathfrak{t} = \R^d$ and we take the same definition of the set $A$, this time with a different choice of $\mathfrak{s}$. 
We replace the sets $B_a$ with $\bigcap_k B_a^k$ where $B_a^k = (a+\mathfrak{t})\cap Eq(S_k)$.
Using the analogously constructed subspaces $\mathfrak{s}^k$  each of the sets $B_a^k$ is of full measure, whence $\bigcap_k B_a^k$ is of full measure and of full Hausdorff dimension.  The rest of the proof follows as above.  
\end{proof}

\section{Quasihyperbolic toral endomorphisms}

In this section we develop the ideas introduced in the previous section to extend the result to certain quasihyperbolic endomorphisms of the torus.  Let $S$ and $T$ be endomorphisms of $\T^d$ such that no eigenvalue of either is zero or a root of unity.  Now we have the decomposition
\[
\R^d = \mathfrak{s}_- \oplus \mathfrak{s}_0 \oplus \mathfrak{s}_+ = \mathfrak{t}_- \oplus \mathfrak{t}_0 \oplus \mathfrak{t}_+
\]
where $\mathfrak{s}_-$ is the subspace corresponding to the eigenvalues of $S$ of modulus less than 1, $\mathfrak{s}_0$ is the subspace corresponding to eigenvalues of $S$ of modulus equal to 1, and $\mathfrak{s}_+$ to those of modulus greater than 1, and similarly for $\mathfrak{t}_-, \mathfrak{t}_0, \mathfrak{t}_+$ and $T$.  Again we assume that neither of $\mathfrak{s}_-$ nor $\mathfrak{t}_-$ are trivial.  In the previous section we assumed that $\mathfrak{s}_-$ and $\mathfrak{t}_-$ spanned $\R^d$, but in this section our assumption is more complicated.  Namely, we require that the subspaces $\mathfrak{s}_-\oplus \mathfrak{s}_0'$ and $\mathfrak{t}_-\oplus \mathfrak{t}_0'$ span $\R^d$ where $\mathfrak{s}_0'$ and $\mathfrak{t}_0'$ are the sums of eigenspaces of eigenvalues of modulus 1 for $S$ and $T$, respectively.  The nature of this restriction will become apparent later in the section.  Choose nontrivial subspaces $\mathfrak{s} \subset \mathfrak{s}_-\oplus \mathfrak{s}_0'$ and $\mathfrak{t} \subset \mathfrak{t}_-\oplus \mathfrak{t}_0'$ such that 
\begin{equation*}\label{Eqn:Decomp}
\R^d = \mathfrak{s} \oplus \mathfrak{t}.
\end{equation*}

As in the previous section, we want to make use of the Marstrand slicing theorem.  In order to do this, we must demonstrate that sets of the form $x + \mathfrak{s}$ are contained entirely in or are disjoint from $Eq(S)$, and similarly for $x + \mathfrak{t}$ and $ND(T)$.  For this we need an analogue of Lemma 3.1 which states that for any $x\in \T^d$ and $y\in \mathfrak{s}$, the point $x+y$ has the same asymptotic behavior as $x$ (that is, either have nondense or equidistributed orbit) under $S$.  Let us first focus on the property of equidistribution.  Consider the action of $S$ on the space $\mathfrak{s}_0$.  As explained by Lind \cite{L} it can be realized as the action of the block matrix 
\[
R =
\begin{pmatrix}
J(R_1, n_1)& & \\
	& \ddots & \\
	&	& J(R_m, n_m)
\end{pmatrix}
\]
where each $J(R_i,n_i)$ is a Jordan block consisting of $n_i$ copies of the $2\times 2$ rotation matrix $R_i$, 
\[
J(R_i,n_i) = \begin{pmatrix} R_i & I& \\
						&   \ddots &\ddots \\
						&	& \ddots &I \\
						&            &     & R_i    \end{pmatrix}.
\]
Observe that when such a Jordan block is applied to a vector $y = (y_1, y_2, \dots, y_{n_i})^\top \in \R^{2n_i}$,
\[
J(R_i, n_i) y = \begin{pmatrix}
R_i y_1 + y_2\\
R_i y_2 + y_3 \\
\vdots \\
R_i y_{n-1} + y_n \\
R_i y_n
\end{pmatrix},
\]
there is drift or shearing in all but the last coordinates.  We are unable to control this drift and so cannot easily deduce the analogue of Lemma 3.1 for the full subspace $\mathfrak{s}_-\oplus\mathfrak{s}_0$.  This is why we instead insist on working with the subspace $\mathfrak{s} = \mathfrak{s}_- \oplus \mathfrak{s}_0'$.  On $\mathfrak{s}_0'$, the action of $S$ is realized by the matrix
\[
R = \begin{pmatrix}
R_1 & & \\
   & \ddots & \\
   &   & R_m
\end{pmatrix}.
\]
(When $\mathfrak{s}_0' =\mathfrak{s}_0$ $S$ is said to be \emph{central spin}, in which case the restriction of $S$ to $\mathfrak{s}_0$ is always of the latter form.)

Let us look in closer detail at the actions of $R$ and $S$ on $\mathfrak{s}_0' \simeq \R^{2m}$ and $\T^d$, respectively.   In particular, let us determine the spectral types of $R$ and $S$.  Recall that for a unitary action $U$ on a separable Hilbert space $\mathcal{H}$ there is a decomposition $\mathcal{H} = \bigoplus_i Z(x_i)$ of $\mathcal{H}$ into cyclic subspaces along with a sequence of measures $\sigma_{x_1} \gg \sigma_{x_2} \gg \cdots$ on the unit circle such that Fourier transforms of these measures are defined by $\widehat{\sigma}_{x_i}(n) = \langle U^n x_i, x_i\rangle$.  These measures are called the spectral measures of $x_i$ with respect to $U$, and the measure class of $\sigma_{x_1}$ is called the spectral type of $U$.  (For further details on spectral types we refer to Glasner \cite{G}.)  

Now $R$ is not an integer matrix and so does not act on $\T^d$ directly.  However $R$ does act by rotation on each factor of $\R^2 \oplus \cdots \oplus \R^2 = \R^{2m}$ and can be thought of as irrational rotation by a vector ${\bf \alpha}$ on $\T^{m}$.  This gives rise to a unitary action $U_R$ on $L_0^2(\T^m)$ with $U_R(f) = f \circ R$.  It is well known that the nontrivial  characters of $\T^m$ , namely $f_{\bf j}(x) = e^{2\pi i \bf j \cdot x}$ for $\bf j \in \Z^m$ with ${\bf j} \neq 0$, span a dense subspace of $L_0^2(\T^m)$.  Now any character $f_{\bf j}$ is an eigenfunction of $U_R$ with eigenvalue $e^{2\pi i {\bf j} \cdot {\bf \alpha}}$.  Thus $R$ has purely discrete spectrum, and the spectral type of $R$ is an atomic measure.

On the other hand, $S$ has continuous spectral type.  Indeed, consider the spectral measure $\sigma_{\bf j}$ associated to any nontrivial character $f_{\bf j}$ for ${\bf j} \in \Z^d$ of $\T^d$ and the action $U_S$ on $L_0^2(\T^d)$ induced by $S$.  We may compute the Fourier transform of $\sigma_{\bf j}$ as 
\begin{align*}
	\widehat{\sigma}_{\bf j}(n) &= \langle U_S^n f_{\bf j}, f_{\bf j} \rangle\\
				& = \int_{\T^d} e^{2\pi i {\bf j} \cdot S^n x} e^{-2\pi i {\bf j} \cdot x } \\
				& =  \int_{\T^d} e^{2\pi i {\bf j} \cdot (S^n-\Id) x},
\end{align*}
where $\Id$ is the $d\times d$ identity matrix.  
For $n = 0$, we see that $\widehat{\sigma}_{\bf j} ( n) = 1$.  Now take $n>0$. Observe that $S^n-\Id$, by assumption, has trivial kernel.  Thus, ${\bf j} \cdot (S^n - \Id) x = j_1 n_1 x_1 + \dots + j_d n_d x_d$ for some integers $n_i$ with $n_1 x_1 + \dots + n_d x_d$ nonzero for all nonzero $x$.  Clearly, $j_1 n_1 x_1 + \dots + j_d n_d x_d$ is identically zero only when ${\bf j} = 0$, a situation we have excluded.  Therefore, for $n>0$, $\widehat{\sigma}_{\bf j }(n) = 1$ if $n=0$ and $\widehat{\sigma}_{\bf j }(n) = 0$ otherwise.  Hence $\sigma_{\bf j}$ is the Lebesgue measure, and $S$ has continuous spectral type.  These distinct spectral types of $S$ and $R$ will be an important ingredient in the next lemma. As a result of our discussion above we know that the spectral types of $S$ and $R$ are mutually singular.

We remark here that the action associated to the restriction of $S$ to the sum $\mathfrak{s}_0$ of the generalized eigenspaces with eigenvalues of modulus 1 has mixed spectrum, yet still has Lebesgue spectral type (in the case that $S$ is not central spin), and so is not, in this way, distinguishable from the full action of $S$ on $\T^d$.  

We will use one more tool to prove the next lemma and that is the structure of joinings of measure preserving systems.  Let $X$ and $Y$ be measure spaces with probability measures $\mu$ and $\nu$, respectively.  Suppose $A$ is a measure preserving and ergodic transformation of $X$ and that $B$ is a measure preserving and ergodic transformation of $Y$.  Then we may consider the product space $X\times Y$ with the product action $A \times B$.  Denote by $\pi_X: X\times Y \to X$ and $\pi_Y: X\times Y \to Y$ the canonical projections.  A joining of $X$ and $Y$ is an $A\times B$-invariant probability measure $\lambda$ on $X\times Y$ with the property that $\pi^*_X \lambda = \mu $ and $\pi^*_Y \lambda = \nu $.  The product measure $\mu \times \nu$ is clearly a joining of $X$ and $Y$,  however there may be others. If $\mu \times \nu$ is the only joining of $X$ and $Y$, then $X$ and $Y$ are called disjoint.  The following lemma follows almost immediately from the definitions.

\begin{lemma}\label{Lem:DisjointEqui}
Given two measure preserving and ergodic probability spaces $(X, A, \mu)$ and $(Y, B, \nu)$ which are disjoint and two points $x\in X$ and $y \in Y$ which equidistribute under $A$ and $B$, respectively, then $(x,y)$ equidistributes  in $X\times Y$ under  $A \times B$ with respect to the product measure $\mu \times \nu$.  
\end{lemma}

\begin{proof}
Define a probability measure $\lambda_N$ on $X \times Y$ as the average of point masses:
\[
\lambda_N = \frac{1}{N}\sum_{n=0}^{N-1} \delta_{(A^n x, B^n y)}.
\]
Let $\lambda$ be a weak* limit of the sequence $\lambda_N$.  Then $\lambda$ is $A\times B$-invariant from construction.  Let $f\in C_c(X)$, and define $F\in C_c(X\times Y)$ to be $F(x,y) = f(x)$.  Then
\begin{align*}
\int f d\pi^*_X\lambda &= \int F d\lambda = \lim_N \int F d\lambda_N \\
&= \lim_N \frac{1}{N}\sum_{n=0}^{N-1} F(A^n x, B^n y) = \lim_N \frac{1}{N} \sum_{n=0}^{N-1}f(A^nx)  = \int f d\mu,
\end{align*}
and so $\pi^*_X\lambda = \mu$.  Similarly, $\pi^*_Y\lambda = \nu$.  Therefore $\lambda$ is a joining of $X$ and $Y$, and $\lambda = \mu \times \nu$, whence $(x,y)$ equidistributes under $A\times B$.  
\end{proof}

The relationship between spectral types and joinings that will be of use to us is the following proposition. For a proof, see \cite{G}:
\begin{proposition}
For two probability spaces $X$ and $Y$ with ergodic actions $A$ and $B$, respectively, if the spectral types of the corresponding $L_0^2$-actions are mutually singular, then $X$ and $Y$ are disjoint.
\end{proposition}

We are now ready to prove the analogue of the equidistribution statement of Lemma 3.4.  
\begin{lemma}\label{Lem:EquiFoli}
Let $x \in \T^d$ and let $y\in \mathfrak{s}$.  Then $x$ equidistributes under $S$ if and only if $x+y$ equidistributes under $S$. The analoguous statement holds for the action of $T$ and $y \in \mathfrak{t}$.
\end{lemma}

\begin{proof}
Let $x\in \T^d$ and $y\in \mathfrak{s}$ with $y\neq 0$ (otherwise the statement is a tautology). Suppose that $x$ equidistributes in $\T^d$.  We wish to show that $x+y$ equidistributes as well. The result for $y$ in the contracting direction of $S$ was already shown in Section 3, so without loss of generality we will assume that $y$ is in the central eigenspace of $S$ and that $S$ acts on $y$ as the irrational rotation $R:\T^m \to \T^m$ described above. Now the orbit of $y$ under $R$ is dense in a copy of $\T^m$ in $\mathfrak{s}$.  In fact, as $R$ is an irrational rotation and so is uniquely ergodic, $y$ equidistributes in $\T^m$ under $R$.  In the last few paragraphs, it has been established that the action of $S$ on $\T^d$ and the action of $R$ on $\T^m$ are disjoint.  Since we may identify the action of $S$ on $x+y$ with the action of $S\times R$ on $(x,y)$, we see that by Lemma~\ref{Lem:DisjointEqui}, $x+y$ equidistributes under $S$.  
\end{proof}

At this point we are in a position to supply the proof of the following weaker version of Theorem~\ref{Thm3}.

\begin{theorem}\label{Thm2}
Let $\T^d =\R^d / \Z^d$ be the $d$ dimensional torus.  Let $S$ and $T$ be quasihyperbolic endomorphisms of $\T^d$. Denote by $\mathfrak{s}_-$ (resp. $\mathfrak{t}_-$) the subspace contracted by $S$ (resp. by $T$) and denote by $\mathfrak{s}_0'$ (resp. $\mathfrak{t}_0'$) the sum of the eigenspaces of $S$ (resp. of $T$) of eigenvalue of modulus 1.  Assume that $\R^d$ is spanned by $\mathfrak{s}_-\oplus \mathfrak{s}_0'$ and $\mathfrak{t}_-\oplus \mathfrak{t}_0'$ and that both of these subspaces are nontrivial.  Then 
\[
\dim( Eq(S)  \cap NEq(T)) = d.
\]
\end{theorem}

\begin{proof}[Proof of Theorem~\ref{Thm2}]
Lemma~\ref{Lem:EquiFoli} tells us that any set of the form $x+\mathfrak{s}$ is either contained in or disjoint from $Eq(S)$.  Similarly, any set of the form $x+\mathfrak{t}$ is either contained in or disjoint from $NEq(T)$.  Thus we may write 
\[
Eq(S) = \bigcup_{x\in\mathfrak{t}\cap Eq(S)}x+\mathfrak{s} \text{ and } NEq(T) = \bigcup_{x\in\mathfrak{s}\cap NEq(T)}x+\mathfrak{t}.
\]
The same proofs as given in Lemmas 3.2 and 3.4 show that $\mathfrak{t}\cap Eq(S)$ has full measure in $\mathfrak{t}$ and $\mathfrak{s}\cap NEq(T)$ is winning in $\mathfrak{s}$ (where we note that $NEq(T)$ is winning because it contains the winning set $ND(T)$).  Furthermore, both versions of the proof of Theorem~\ref{Thm1} may be applied in the current situation to finish the proof.
\end{proof}

To prove Theorem~\ref{Thm3}, we must work a bit harder to show the analogue of Lemma 3.4 for points with nondense orbit under the transformation $T$.  To prove that for any $x\in ND(T)$ and $y\in\mathfrak{t}$, the point $x+y$ also has nondense orbit, we must introduce and use topological entropy.  (See \cite{W} for complete details.)

Let $X$ be a compact topological space and let $A$ be a continuous self-map of $X$.  We recall the definition of the topological entropy of $A$.  For two open covers $\mathcal{U}$ and $\mathcal{V}$ of $X$ define their common refinement (or join) $\mathcal{U} \vee \mathcal{V}$ to be the open cover consisting of sets of the form $U \cap V$ for $U\in\mathcal{U}$ and $V\in\mathcal{V}$.  For an open cover $\mathcal{U}$ of $X$, let $N(\mathcal{U})$ be the number of elements in a minimal subcover of $X$.  We define $H(\mathcal{U}) = \log N(\mathcal{U})$.  The topological entropy of $A$ relative to $\mathcal{U}$ is 
\[
h(A, \mathcal{U}) = \lim_{n\to\infty} \frac{1}{n} H(\bigvee_{i=0}^{n-1}A^{-i}\mathcal{U}).
\]
This limit always exists by a lemma of Fekete.  Finally, define the topological entropy of $A$ to be
\[
h(A) = \sup_{\mathcal{U}} h(A,\mathcal{U}),
\]
where the supremum ranges over all open covers of $X$.  

\begin{lemma}
For a closed subset $Y\subseteq X$ with $AY \subseteq Y$, 
\[h(A|_Y)\leq h(A).\]
\end{lemma}

\begin{proof}
Let $\mathcal{U}$ be an open cover of $Y$.  By definition of the subspace topology,  each element $U\in \mathcal{U}$ is of the form $U = \widetilde{U}\cap Y$ for some open set $\widetilde{U}$ of $X$.  Set $\widetilde{\mathcal{U}}$ to be the collection consisting of these sets $\widetilde{U}$ along with any other open set completing $\widetilde{\mathcal{U}}$ to a cover of $X$.  Thus, $N(\mathcal{U})\leq N(\widetilde{\mathcal{U}})$.  Since $AY \subseteq Y$, $A|_Y^{-1}(\mathcal{U})$ is again an open cover of $Y$, as is $A^{-1}(\widetilde{\mathcal{U}})$ of $X$.  Moreover, elements of $A|_Y^{-1}(\mathcal{U})$ are of the form
\[
A|_Y^{-1}(U) = A^{-1}(U) \cap Y = A^{-1}(\widetilde{U}\cap Y)\cap Y = A^{-1}(\widetilde{U})\cap Y.
\]
Thus $N(\mathcal{U}\vee A|_Y^{-1}(\mathcal{U})) \leq N(\widetilde{\mathcal{U}}\vee A^{-1}(\widetilde{\mathcal{U}}))$.  Inductively, we arrive at the inequality $N(\bigvee_{i=0}^{n-1} A|_Y^{-i}(\mathcal{U})) \leq N(\bigvee_{i=0}^{n-1} A^{-i}(\widetilde{\mathcal{U}}))$.  By convexity of logarithm and continuity of limits, we see $h(A|_Y,\mathcal{U})\leq h(A,\widetilde{\mathcal{U}})$.  Since each open cover of $Y$ induces an open cover of $X$, we have that $h(A|_Y) \leq h(A)$.  
\end{proof}

In the specific case of ergodic toral automorphisms, we have a more specific version of Lemma 4.5.  The next lemma states that the density of an orbit can be measured by entropy.  It is in this lemma, where a result of Berg is applied, that it is essential that $T$ is an automorphism.  Denote by $O(x)$ the closure of the orbit $\{T^n x\}_{n\in \N}$ of $x$ under $T$.  

\begin{lemma}
For $x\in\T^d$, $O(x)\subsetneq \T^d$ if and only if $h(T|_{O(x)})< h(T)$.
\end{lemma}
\begin{proof}
By Lemma 4.5, for any $x\in \T^d$, we have $h(T|_{O(x)})\leq h(T)$.  If $x\in D(T)$, then $O(x) = \T^d$ and $h(T|_{O(x)}) = h(T)$.  Suppose, then, that $x\in ND(T)$.  We want to show $h(T|_{O(x)}) < h(T)$; assume otherwise. By the variational principle, there exists a sequence of probability measures $\mu_n$ supported on $O(x)$ with $h_{\mu_n}(T|_{O(x)})$ increasing to $h(T|_{O(x)})=h(T)$.  (Here $h_\nu(T)$ denotes the metric entropy of $T$ with respect to a probability measure $\nu$, and the variational principle states that $h(T) = \sup_\nu h_\nu(T)$.  See \cite{W} for the definition of metric entropy and details of the variational principle.)  Let $\mu$ be a weak* limit of the sequence $\mu_n$.  Since $h_\nu(T)$ is an upper semi-continuous function of $\nu$ (see Theorem 4.1 of \cite{Newhouse}), $h_\mu(T) = h(T|_{O(x)})=h(T)$.  By \cite{Berg}, the unique measure of maximal metric entropy for the ergodic toral automorphism $T$ is the Lebesgue measure on $\T^d$, whence $\mu = m$.  However, $\mu$ is supported on $O(x) \neq \T^d$.  Thus we arrive at a contradiction, and so $h(T|_{O(x)}) < h(T)$.  
\end{proof}

Now we will use these lemmas to prove the analogue of Lemma 3.4 for nondense orbits.  

\begin{lemma}
Let $x\in \T^d$ and let $y\in \mathfrak{t}$.  Then $x$ has nondense orbit under $T$ if and only if $x+y$ has nondense orbit under $T$.
\end{lemma}

\begin{proof}
Assume $x\in ND(T)$, so that $O(x)\subsetneq \T^d$ and  by the previous lemma $h(T|_{O(x)}) < h(T)$.  We will show that $h(T|_{O(x+y)}) \leq h(T|_{O(x)})$.  By the previous lemma and symmetry, we will have proven the desired result.  

The result for $y$ in the contracting direction of $T$ was already shown in Section 3, so without loss of generality we will assume that $y$ is a nonzero element in the central eigenspace of $T$ and that $T$ acts on $y$ by the irrational rotation $R:\T^m \to \T^m$ described above.  The system $(O(x), T|_{O(x)})$ is a factor of the system $(O(x) \times \T^m, T|_{O(x)}\times R)$ in the natural way.  It is a basic fact of topological entropy that $h(T|_{O(x)}\times R) = h(T|_{O(x)})+ h(R)$.  Moreover, since $R$ is a rotation $h(R) = 0$.  (See \cite{W}.)  Hence $h(T|_{O(x)}\times R) = h(T|_{O(x)})$.  Now, $O(x+y)$ is in a natural manner included in $O(x) \times \T^d$ as a closed $T\times R$ invariant set.  Applying Lemma 4.4, we have that $h(T|_{O(x+y)}) \leq h(T|_{O(x)})$, finishing the proof.
\end{proof}

All of the pieces are in place to prove Theorem~\ref{Thm3}.
\begin{proof}[Proof of Theorem~\ref{Thm3}]
Lemma~\ref{Lem:EquiFoli} tells us that any set of the form $x+\mathfrak{s}$ is either contained in or disjoint from $Eq(S)$.  Similarly, Lemma 4.6 states any set of the form $x+\mathfrak{t}$ is either contained in or disjoint from $ND(T)$.  Thus we may write 
\[
Eq(S) = \bigcup_{x\in\mathfrak{t}\cap Eq(S)}x+\mathfrak{s} \text{ and } ND(T) = \bigcup_{x\in\mathfrak{s}\cap ND(T)}x+\mathfrak{t}.
\]
The same proofs as given in Lemmas 3.5 and 3.7 show that $\mathfrak{t}\cap Eq(S)$ has full measure in $\mathfrak{t}$ and $\mathfrak{s}\cap ND(T)$ is winning in $\mathfrak{s}$.  Furthermore, both versions of the proof of Theorem~\ref{Thm1} may be applied in the current situation to finish the proof.
\end{proof}

The proofs of Corollaries~\ref{Cor3} and \ref{Cor4} follow exactly the same line of reasoning given in Section 3 for Corollaries~\ref{Cor1} and \ref{Cor2}.

\end{document}